\documentclass[11pt,reqno]{amsart}

\usepackage{graphicx,enumerate, stmaryrd, color, url}
\usepackage{amsmath}
\usepackage{amssymb}
\usepackage{amsthm}
\usepackage{amsfonts}
\usepackage{epsfig,amscd,amsthm,mathtools,fourier}
\usepackage{tikz-cd}
\usepackage{tikz}
\usepackage{fullpage}
\usepackage[all]{xy}
\usepackage{comment}
\usepackage{times}
\usepackage{bm}
\usepackage[hidelinks]{hyperref}
\usepackage{appendix}
\usepackage{float}
\usepackage[margin=1.3in,marginparwidth=1.1in,marginparsep=0.1in]{geometry}
\usepackage{marginnote}
\usepackage{import}
\usepackage[normalem]{ulem}

\newtheorem{thm}{Theorem}[section]

\newtheorem{lma}[thm]{Lemma}
\newtheorem{prop}[thm]{Proposition}

\theoremstyle{definition}
\newtheorem{defn}[thm]{Definition}
\newtheorem{rmk}[thm]{Remark}

\newtheorem{ex}[thm]{Example}

\newtheorem{Question}[thm]{Question}

\newcommand{\un}{\underline}

\newcommand{\md}{\underline{d}}

\newcommand{\mdeg}{\underline{\operatorname{deg}}}

\newcommand{\val}{\operatorname{val}}

\newcommand{\cO}{\mathcal{O}}

\newcommand{\Pic}{{\operatorname{{Pic}}}}

\newcommand{\pr}{\mathrm{pr}}

\newcommand{\Img}{\mathrm{Im}}

\newcommand{\indeg}{\mathrm{indeg}}

\newcommand{\GG}{\mathbb{G}}

\newcommand{\ud}{\underline{d}}

\title{On the rank of general linear series on stable curves} 
\author{Karl Christ}

\thanks{The author was supported by the Israel Science Foundation (grant No. 821/16) and by the Center for Advanced Studies at BGU}

\address[Christ]{$~^1$Department of Mathematics\\
	Ben-Gurion University of the Negev\\P.O.Box 653 \\Be'er Sheva\\ 84105\\  Israel and $~^2$Institute of Algebraic Geometry\\
	Leibniz University Hannover\\Welfengarten 1 \\30167 Han\-no\-ver\\  Germany }\email{kchrist@math.uni-hannover.de}

\begin{document}
	
\maketitle

\begin{abstract}
    We study the dimension of loci of special line bundles on stable curves and for a fixed semistable multidegree. In case of total degree $d = g - 1$, we characterize when the effective locus gives a Theta divisor. In case of degree $g - 2$ and $g$, we show that the locus is either empty or has the expected dimension. This leads to a new characterization of semistability in these degrees. In the remaining cases, we show that the special locus has codimension at least $2$. If the multidegree in addition is non-negative on each irreducible component of the curve, we show that the special locus contains an irrreducible component of expected dimension.
\end{abstract}

\section{Introduction}
If $X$ is a smooth curve of genus $g$, then a \emph{general} line bundle $L$ of degree $d$ satisfies \[h^0(X,L) = \max \left \{ 0, d - g + 1 \right \}.\] The locus of \emph{special} line bundles in $\Pic^d(X)$, those that have additional global sections, is empty for $d < 0$ and $d > 2g - 2$, and irreducible of dimension $d$ if $0 \leq d \leq g-1$ and irreducible of dimension $2g-2 -d$ if $g - 1 \leq d \leq 2g - 2$.

In this paper, we are interested in how this picture changes if $X$ is no longer assumed to be
a smooth curve, but allowed to have nodal singularities. More precisely, we assume $X$ to be a \emph{stable} curve. Stable curves give a well-understood compactification of the moduli space of smooth curves. In this way they provide important tools for understanding the geometry of the moduli space, as well as studying smooth curves via degeneration techniques.

Two basic properties remain as in the smooth case: First, the Riemann Roch theorem still holds, and hence also the lower bound $h^0(X,L) \geq d - g + 1$. And second, loci of special line bundles (or, more generally, any Brill Noether locus) can be realized as a degeneracy locus of a map between vector bundles. Such a locus is either empty, or each of its irreducible components has at least the expected dimension.

If $X$ is reducible, the similarities with the case of smooth curves do not go much further. More precisely, in this case also the degree $d$ Picard scheme $\Pic^d(X)$ is no longer irreducible. To be able to talk about generic behaviour, one thus needs to restrict to an irreducible component of $\Pic^d(X)$. We denote such an irreducible component by $\Pic^{\md}(X)$, parametrizing line bundles on $X$ of fixed \emph{multidegree} $\md$. A multidegree is a tuple of integers, one for each irreducible component $X_v$ of $X$, that prescribes the degree of the restriction of a line bundle to $X_v$. 

For almost all choices of multidegree $\md$ of fixed total degree $d$ none of the statements for smooth curves in the first paragraph remains true if $X$ is reducible. In fact, for any integer $r$ and fixed total degree $d$, there are only finitely many multidegrees $\md$ of total degree $d$ such that there exists a line bundle $L$ of multidegree $\md$ and with $h^0(X,L) \leq r$ -- whereas there are infinitely many multidegrees of fixed total degree $d$.

To remedy this, we focus in this paper on \emph{semistable multidegrees}, as introduced by Caporaso to construct a universal compactified Jacobian over the moduli space of stable curves \cite{Caporasocompactification}. See Definition~\ref{def:semistable}.  Line bundles with semistable multidegree are those that are slope semistable with respect to the dualizing sheaf $\omega_X$ of $X$ by \cite[\S 1]{Alexeev}.

Our choice of multidegrees is motivated by two previous results. First, as in the case of smooth curves, outside of the range $0 \leq d \leq 2g - 2$ \emph{every} line bundle with semistable multidegree is non-special by \cite[Theorem 2.3]{caporasosemistable}. And second, the case of $d = g-1$ allows for a theory of Theta divisors, as we explain below. Even more, one can characterize semistability for degree $g - 1$ in these terms. We note however, that in other regards there are better-behaved classes of multidegrees, and there seems to be no overall best-behaved choice. See for example \cite{CLM} and \cite{chr2} for the question of an upper bound on $h^0(X,L)$ for \emph{all} line bundles $L$ of fixed multidegree. 

\subsection{Results}
Suppose $\md$ is a multidegree of total degree $d \leq g - 1$. Then the expected dimension of the effective locus \[W_{\md}(X) \coloneqq \left\{[L] \in \Pic^{\md}(X) \mid h^0(X,L) \geq 1\right\}\] is $d$. In particular, if $d = g-1$, $W_{\md}(X)$ is expected to be a divisor in $\Pic^{\md}(X)$, called the \emph{Theta divisor}. 

Since $W_{\md}(X)$ can be realized as a degeneracy locus, there are three options for the actual dimension of $W_{\md}(X)$ in degree $g -1$: $W_{\md}(X)$ is either empty, a divisor in $\Pic^{\md}(X)$, or all of $\Pic^{\md}(X)$. Thus to show that $W_{\md}(X)$ is a divisor, it suffices to show that there exist both an effective and a non-effective line bundle of multidegree $\md$. For the latter, a very pleasing answer was found by Beauville \cite[Lemma 2.1]{Beauville}. Namely, a multidegree $\md$ of total degree $g - 1$ is semistable if and only if there exists a non-effective line bundle of multidegree $\md$.

In this paper, we first settle the remaining part of the question, namely when the effective locus $W_{\md}(X)$ is empty. It turns out, that $W_{\md}(X)$ indeed can be empty, even if $\md$ is semistable of degree $g - 1$. Example~\ref{ex:theta_empty} gives instances for arbitrary $g \geq 2$. Before we can characterize when this happens, we need to recall one particularly convenient combinatorial description of semistable multidegrees of total degree $g - 1$.

We denote by $\GG_X$ the dual graph of $X$. To an orientation $O$ of the edges of $\GG_X$, one can associate a multidegree $\md_O$. Namely, the value of $\md_O$ on an irreducible component $X_v$ of $X$ is given by \[\mathrm{indeg}_O(v) - 1 + g_v.\] Here $v$ is the vertex of $\GG_X$ corresponding to $X_v$, $\mathrm{indeg}_O(v)$ denotes the number of incoming edges at $v$ in the orientation $O$, and $g_v$ denotes the geometric genus of $X_v$. As observed in \cite{Beauville} and \cite{Alexeev}, a multidegree $\md$ is semistable of total degree $g - 1$ if and only if $\md = \md_O$ for some orientation $O$. As a final ingredient, recall that a directed cycle in an orientation $O$ is a cycle in which every vertex is adjacent to one incoming and one outgoing edge of the cycle.

\begin{thm}\label{thm:main1}
Let $X$ be a stable curve and $\md$ a multidegree of total degree $g - 1$. Then the effective locus $W_{\md}(X) \subset \Pic^{\md}(X)$ has codimension $1$ if and only if $\md$ is semistable and either $X$ contains an irreducible component that is not rational, or the orientation giving $\md$ contains a directed cycle. 
\end{thm}

See Theorem~\ref{thm:theta divisor}. One can view this as a small correction to \cite[Proposition 2.2]{Beauville}, in which it is claimed that $\md$ is semistable if and only if a Theta divisor exists. The possibility of an empty effective locus however is not considered. 

The existence of a Theta divisor is a foundational result in the theory of compactified Jacobians and generalized Prym varieties, since a Theta divisor allows to define a canonical polarization on these varieties. See, for example, \cite{Alexeev}, \cite{caporasotheta}, \cite{CV}, and \cite{Beauville}, \cite{ABH},  \cite{MGHL}). In these cases, the multidegrees $\md$ do satisfy the second condition of Theorem~\ref{thm:main1} and no issue arises. See Remark~\ref{rmk:correction old statement} for a detailed discussion.

\medskip

We use Theorem~\ref{thm:main1} and combinatorics of semistable multidegrees to obtain new results for semistable multidegrees of total degrees other than $g - 1$. Recall that a multidegree $\md$ itself is called \emph{effective}, if it is non-negative on each irreducible component of $X$.

\begin{thm} \label{thm:main2}
Let $X$ be a stable curve and $\md$ a semistable multidegree of total degree $d \leq g - 2$. Then each irreducible component of the effective locus $W_{\md}(X)$ has dimension at most $g - 2$.
If $\md$ is in addition effective, then the effective locus  $W_{\md}(X)$ contains an irreducible component of expected dimension, $d$. 
\end{thm}

See Theorem~\ref{thm:general effective loci}. It is not hard to see, that a line bundle $L$ is semistable or special, if and only if its residual $\omega_X \otimes L^{-1}$ is so (cf. Remark~\ref{rmk:residual}). Hence Theorem~\ref{thm:main2} also gives the analogous statements for special loci in case $d \geq g$.

The irreducible component of expected dimension in the last statement of Theorem~\ref{thm:main2} is the image $A_{\md}(X)$ of the rational Abel map $\alpha_{\md}$. This map is defined whenever $\md$ is effective, as follows:
\[
\alpha_{\md}\colon \prod_{v} \left ( X_v^{\mathrm{sm}} \right)^{\md_v} \to \Pic^{\md}(X), \, (p_1, \dots, p_d) \mapsto \mathcal O_X\left(p_1 + \ldots + p_d \right).
\]
Here $X_v^{\mathrm{sm}}$ denotes the intersection of $X_v$ with the smooth locus of $X$ and $\md_v$ the value of the multidegree $\md$ on $X_v$. In general, the dimension of the image $A_{\md}(X)$ can be smaller than $d$ even if $d \leq g$, and $A_{\md}(X)$ need not be an irreducible component of the effective locus. See Remark~\ref{rmk:abel}. The claim of the theorem is, that both these properties however do hold when $\md$ is semistable.

Notice next, that Theorem~\ref{thm:main2} shows that the effective locus is either empty or has the expected dimension, when $\md$ is semistable of total degree $d = g - 2$. As opposed to the  case $d = g - 1$, this no longer characterizes semistability. See Examples~\ref{ex:g-2 1} and \ref{ex:g-2 2}. On the other hand, we obtain a new characterization of semistability in degree $g - 2$ in Lemma~\ref{lma:char deg g-2}. Since it allows for a more elegant formulation, we state here the residual version for total degree $g$, as in Theorem~\ref{thm:char deg g}.

\begin{thm} \label{thm:dominant abel}
Let $X$ be a stable curve and $\md$ a multidegree of total degree $g$. Then $\md$ is semistable if and only if $\md$ is effective and the rational Abel map $\alpha_{\md}$ is dominant.   
\end{thm}

As mentioned above, it is well-known that semistability in degree $g - 1$ shows many surprising interactions, among others with combinatorics and the existence of Theta divisors. In addition to \cite{CC} and \cite{CPS}, Theorem~\ref{thm:dominant abel} provides further evidence that semistability in degree $g$ exhibits similarly special behaviour.

Finally, for $d < g- 2$ the effective locus can have larger than expected dimension, even if $\md$ is semistable. This can be the case already for $d = g - 3$ as in Example~\ref{ex: g-3}. 

\subsection{Structure of the paper} In Section~\ref{sec: preliminaries}, we fix some notation. In Section~\ref{sec:semistable}, we recall basic properties of semistable and special line bundles. In Section~\ref{sec:theta}, we discuss the case $d = g -1$ and the existence of Theta divisors. To do so, we recall the description of semistability in terms of orientations in Subsection~\ref{subsec:orientations}. After establishing existence of Theta divisors in Subsection~\ref{subsec:theta}, we recall a  description of the irreducible components of Theta divisors due to Coelho and Esteves~\cite{coelhoesteves} in Subsection~\ref{subsec:abel}. In Section~\ref{sec:g-2}, we discuss the case $d = g- 2$ and its residual case $d = g$. We show that in these cases the special loci have expected dimension if $\md$ is semistable in Subsection~\ref{subsec:dimension g} and characterize semistability in Subsection~\ref{subsec:semistability g}. Finally, we discuss in Section~\ref{sec:general} the general case and prove Theorem~\ref{thm:main2}.
\medskip

\noindent 
{\bf Acknowledgements.} Many discussions with Lucia Caporaso and Ilya Tyomkin helped shape this paper, and I am very grateful for the insights and suggestions they provided. In addition, I would like to thank Sam Payne and an anonymous referee, whose comments on an earlier draft led to many improvements in presentation.

\section{Notations and Conventions} \label{sec: preliminaries}

Throughout the paper, we work over an algebraically closed field $k$ of characteristic $0$. We consider a curve $X$ over $k$, which we will always assume to be reduced with nodal singularities. If not specified otherwise, we will assume $X$ to be connected.

We denote the \emph{dual graph} of $X$ by $\GG_X$. That is, $\GG_X$ contains a vertex $v$ for every irreducible component $X_v$ of $X$; an edge between vertices $v$ and $w$ for each node in $X_v \cap X_w$, possibly with $v = w$; and each vertex $v$ is assigned the weight $g_v$ given by the geometric genus of $X_v$. In particular, $\GG_X$ may contain multiple edges between the same two vertices, as well as loop edges. We denote by $V(\GG_X)$ and $E(\GG_X)$ the sets of vertices and edges of $\GG_X$, respectively. 

The \emph{genus} of $\GG_X$ is defined as 
\begin{equation}\label{eq:def genus}
    g(\GG_X) = 1 - \chi(\GG_X) + \sum_{v \in V(\GG_X)} g_v,
\end{equation}
where $\chi(\GG_X)$ is the Euler characteristic of $\GG_X$. That is, $1 - \chi(\GG_X) = 1 - |V(\GG_X)| + |E(\GG_X)|$.  The genus of $\GG_X$ equals the arithmetic genus $g(X)$ of $X$. We write $g \vcentcolon= g(X) = g(\GG_X)$ if $X$ is clear from the context. The \emph{valence} $\val(v)$ of a vertex $v$ is the number of edges adjacent to $v$ with loops counted twice.

For a subcurve $Y \subset X$ we write $Y^c = \overline{X \setminus Y}$
for the closure of the complement in $X$. In particular, $Y \cap Y^c$ is a finite union of nodes. Any subcurve $Y \subset X$ corresponds to an \emph{induced subgraph} $\GG_Y$ of $\GG_X$, that is, a subgraph that contains all edges of $\GG_X$ between vertices contained in $\GG_Y$. We denote by $\GG_Y^c$ the dual graph of $Y^c$.  

We denote by $\omega_X$ the \emph{dualizing sheaf} of $X$. It has total degree $2g - 2$. The restriction of $\omega_X$ to a subcurve $Y$ of $X$ has total degree $2 g(Y) - 2 + |Y \cap Y^c|$. In particular, the restriction of $\omega_X$ to an irreducible component $X_v$ has degree $2g_v - 2 + \val(v)$.
A curve $X$ is \emph{stable} if it is connected, $g(X) \geq 2$ and whenever an irreducible component $X_v$ is smooth and rational, then $|X_v \cap X_v^c| \geq 3$. Equivalently, $X$ is stable if it is connected and $\omega_X$ is ample.  

We write $\md$ for a \emph{multidegree}, that is, a formal linear combination of vertices of $\GG_X$ with integer coefficients. In particular, we can add and subtract multidegrees coefficient-wise. We denote by $\md_v$ the coefficient at a vertex $v$ of $\GG_X$. We write $\md \pm v$ for the multidegree obtained by adding or subtracting $1$ from the coefficient at $v$. Any line bundle $L$ on $X$ has an associated multidegree $\mdeg(L)$, defined by $\mdeg(L)_v = \deg\left(L|_{X_v}\right)$. The \emph{total degree} of $\md$ is $\sum_v \md_v$, which coincides with the total degree $\deg(L)$ of $L$ if $\md$ is the multidegree of $L$. A multidegree is \emph{effective}, if $\md_v \geq 0$ for all vertices $v$ of $\GG_X$. 

\section{Semistable multidegrees and Brill Noether loci}
\label{sec:semistable}

In this section, we continue the preliminaries by introducing the main objects of interest for this paper -- semistable multidegrees in Subsection~\ref{subsec:semistable} and special loci in Subsection~\ref{subsec:special loci}.

\subsection{Semistable multidegrees} \label{subsec:semistable}
We begin with semistable line bundles and their multidegrees in the sense of \cite{Caporasocompactification}.

\begin{defn} \label{def:semistable}
Let $X$ be a stable curve. A line bundle $L$ on $X$ of total degree $d$ is called \emph{semistable} if for any subcurve $Y \subset X$ we have
	\begin{equation} \label{eq: upper bound}
	    	 g(Y) - 1 + (d - g + 1) \frac{2 g(Y) - 2 + |Y \cap Y^c|}{2g - 2} \leq \deg(L|_Y).
	\end{equation}
	The line bundle $L$ is called \emph{stable} if the inequality above is strict for every proper subcurve $Y \subsetneq X$.
\end{defn}

Whether $L$ is semistable or stable depends only on its multidegree $\md$, since \[\deg(L|_Y) = \sum_{X_v \subset Y} \md_v.\] We will say that a multidegree $\md$ is semistable or stable, if the corresponding line bundles are so.

\begin{lma}\label{lma:semistable residual}
A line bundle $L$ is semistable or stable if and only if its residual $\omega_X \otimes L^{-1}$ is so.
\end{lma}

\begin{proof}
     Rearranging terms turns Inequality~\eqref{eq: upper bound} into 
    \begin{equation}\label{eq: upper bound classic}
	    	 d \frac{2 g(Y) - 2 + |Y \cap Y^c|}{2g - 2} - \frac{|Y \cap Y^c|}{2} \leq \deg(L|_Y).
    \end{equation}
    We want to apply Inequality~\eqref{eq: upper bound classic} to the subcurve $Y^c = \overline{X \setminus Y}$. Before doing so, note that $\deg(L|_{Y^c}) = d - \deg(L|_{Y})$ and recall that the degree of the restriction of the dualizing sheaf $\omega_X$ to $Y$ is $2 g(Y) - 2 + |Y \cap Y^c|$. In particular, $2 g(Y^c) - 2 + |Y \cap Y^c| = 2g - 2 - \left(2 g(Y) - 2 + |Y \cap Y^c|\right)$. We obtain from \eqref{eq: upper bound classic} applied to $Y^c$:
    \begin{eqnarray*}
        & & d \frac{2 g(Y^c) - 2 + |Y \cap Y^c|}{2g - 2} - \frac{|Y \cap Y^c|}{2} \leq \deg(L|_{Y^c}) \\ 
        &\Leftrightarrow& d - d \frac{2 g(Y) - 2 + |Y \cap Y^c|}{2g - 2} - \frac{|Y \cap Y^c|}{2} \leq d - \deg(L|_{Y}) \\
        &\Leftrightarrow& (2g - 2 - d) \frac{2 g(Y) - 2 + |Y \cap Y^c|}{2g - 2} - \left(2 g(Y) - 2 + |Y \cap Y^c|\right) - \frac{|Y \cap Y^c|}{2} \leq - \deg(L|_{Y}) \\
        &\Leftrightarrow& (2g - 2 - d) \frac{2 g(Y) - 2 + |Y \cap Y^c|}{2g - 2} - \frac{|Y \cap Y^c|}{2} \leq 2 g(Y) - 2 + |Y \cap Y^c| - \deg(L|_{Y}) \\
    \end{eqnarray*}
    The last inequality is Inequality~\eqref{eq: upper bound classic} for $\omega_X \otimes L^{-1}$ and hence the claim follows.
\end{proof}

\subsection{Brill Noether loci} \label{subsec:special loci}

Given a multidegree $\md$ on $X$ 
we write $\Pic(X)$ for the Picard scheme of $X$, and $\Pic^{\md}(X)$ for the connected component of $\Pic(X)$ that parametrizes line bundles of multidegree $\md$. For a line bundle $L$, we denote by $[L]$ the corresponding point of $\Pic(X)$.

\begin{defn}
The \emph{Brill Noether locus} $W_{\md}^r(X)$ is the subset of $\Pic^{\md}(X)$ given by
\[
W_{\md}^r(X) \coloneqq \left\{[L] \in \Pic^{\md}(X) \, | \, h^0(X,L) \geq r + 1 \right\}. 
\]
We will write $W_{\md}(X) \vcentcolon= W_{\md}^0(X)$ for the \emph{effective locus}.
\end{defn}

The next proposition is the standard observation that the Brill Noether locus $W_{\md}^r(X)$ can be realized as a degeneracy locus of a map between vector bundles. We sketch the argument for the convenience of the reader, following the presentation in the proof of \cite[Proposition 2.2]{Beauville}.

\begin{prop}\label{prop:determinantal}
Let $X$ be a stable curve and $\md$ a multidegree of total degree $d$. Then $W_{\md}^r(X)$ is a closed subset of $\Pic^{\md}(X)$ with either $W_{\md}^r(X) = \emptyset$, or each irreducible component of $W_{\md}^r(X)$ has dimension at least $\min \left \{g, g - (r + 1)(g - d + r) \right\}$.  
\end{prop}

Note, in particular, that the effective locus $W_{\md}(X)$ is either empty, or each irreducible component has dimension at least $\min\{d,g\}$.  

\begin{proof}
    Let $\mathcal P \to \Pic^{\md}(X) \times X$ be a Poincar\'e bundle, that is, a line bundle that restricts to $L$ over $[L] \times X$. Denote by $\pr_1$ and $\pr_2$ the projections from $\Pic^{\md}(X) \times X$ to $\Pic^{\md}(X)$ and $X$, respectively. 
    
    Recall that we denote by $\omega_X$ the dualizing sheaf of $X$. We choose $D = \sum_{i = 1}^k p_i$ where the $p_i \in X$ are $k$ smooth points, such that \[h^0\left(X, \omega_X \otimes L^{-1} \otimes \mathcal O_X(D)\right) = \deg\left(\omega_X \otimes L^{-1} \otimes \mathcal O_X(D)\right) - g + 1 =  g - 1 - d + k\] for all $[L] \in \Pic^{\md}(X)$. This is always possible, see \cite[Lemma 2.1]{CF} or \cite[Lemma 2.5]{caporasosemistable}.
    We get a short exact sequence
    \begin{equation*}\label{eq: ses points}
        0 \to \mathcal P \otimes \pr_2^* \mathcal O_X\left(- D\right) \to \mathcal P \to \mathcal P \otimes \pr_2^* \mathcal O_D \to 0. 
    \end{equation*}
    We apply $(\pr_1)_*$ to this sequence and set $E_1 = (\pr_1)_* (\mathcal P \otimes \pr_2^* \mathcal O_D)$, which is locally free of rank $k$, and $E_2 = R^1  (\pr_1)_* \left(\mathcal P \otimes \pr_2^* \mathcal O_X\left(- D\right)\right)$, which is locally free of rank $g - 1 - d + k$ by the choice of $D$. Since $R^1(\pr_1)_* \left(\mathcal P \otimes \pr_2^* \mathcal O_D\right) = 0$, the higher direct image sequence induces an exact sequence 
    \begin{equation*}
        E_1 \xrightarrow[]{u} E_2 \to R^1  (\pr_1)_* \mathcal P \to 0.
    \end{equation*}
    Over $[L] \in \Pic^{\md}(X)$ it restricts to   
     \begin{equation*}
        (E_1)_{[L]} \xrightarrow[]{u_{[L]}} (E_2)_{[L]} \to H^1(X, L) \to 0.
    \end{equation*}
    By the Riemann Roch theorem, $h^0(X,L) \geq r + 1$ if and only if $h^1(X, L) \geq g - d + r$. On the other hand, the exact sequence gives $h^1(X, L) = g - 1 - d + k - \dim\left(\Img(u_{[L]})\right)$. Thus $W_{\md}^r(X)$ is the locus where $u$ has rank at most $k - 1 - r$. The expected dimension of this locus is the Brill Noether number
    \[
    \rho = g - (k - k + 1 + r)(g - 1 - d + k - k + 1 + r) = g - (r + 1)(g - d  + r).
    \]
    The claim now follows by the observation that a degeneracy locus is either empty, the whole space, or each of its irreducible components has dimension at least the expected dimension (see, e.g., \cite[Section 2.4]{curvesv1}).
\end{proof}

We will be interested in loci of special line bundles, which are the first of the Brill Noether loci defined above:
We call a line bundle $L$ of total degree $d$ \emph{special}, if $h^0(X,L) > \max\{0, d - g + 1\}$. The term `special' is of course taken from the case of irreducible curves, where non-special line bundles are dense in $\Pic^d(X)$. 

\begin{lma}\label{lma:special residual}
Let $X$ be a stable curve with dualizing sheaf $\omega_X$. Then a line bundle $L$ on $X$ is special, if and only if its residual $\omega_X \otimes L^{-1}$ is.  
\end{lma}

\begin{proof}
    Suppose $d = \deg(L) \leq g -1$, the case $d \geq g - 1$ is analogous. By the Riemann Roch theorem we have $h^0(X,L) = d - g + 1 + h^0(X, \omega_X \otimes L^{-1})$. Thus \[h^0(X,L) > 0 \Leftrightarrow h^0(X, \omega_X \otimes L^{-1}) > g - d - 1 = (2g -2 - d) - g + 1,\]
    and the claim follows, since $\deg\left(\omega_X \otimes L^{-1}\right) = 2g - 2 - d$.
\end{proof}

\begin{rmk}\label{rmk:residual}
By Lemmas~\ref{lma:semistable residual} and \ref{lma:special residual}, it suffices to restrict to semistable multidegrees of total degree $d \leq g - 1$ and to consider effective loci $W_{\md}(X)$ to obtain the analogous information about special loci also for  $d > g -1$ by passing to residuals. More precisely, if $\md$ is semistable of total degree $d$, then $\md' = \mdeg(\omega_X) - \md$ is semistable of degree $2g - 2 - d$ by Lemma~\ref{lma:semistable residual} and the isomorphism $\Pic^{\md}(X) \to \Pic^{\md'}(X), [L] \mapsto \left[\omega_X \otimes L^{-1}\right]$ preserves special loci by Lemma~\ref{lma:special residual}.
\end{rmk}

\section{Theta divisors} \label{sec:theta}

In this section, we consider the case $d = g - 1$. Semistable multidegrees admit a very convenient description in this case, using orientations on the dual graph. We collect the necessary notions in Subsection~\ref{subsec:orientations}. In Subsection~\ref{subsec:theta}, we discuss existence of Theta divisors, with the main results in Theorem~\ref{thm:theta divisor}. In Subsection~\ref{subsec:abel}, we recall a description of the irreducible components of Theta divisors due to Coelho and Esteves~\cite{coelhoesteves}, that will be useful in the next section.

\subsection{Graph orientations and their associated multidegrees}\label{subsec:orientations}

We begin by fixing some standard notions from graph theory.
An \emph{orientation} $O$ on the graph $\GG_X$ is the assignment of a direction to each edge of $\GG_X$. The multidegree $\md_O$ associated to an orientation $O$ has by definition value \[ \indeg_O(v) -1 + g_v\] on a vertex $v \in V(\GG_X)$, where $\indeg_O(v)$ is the number of edges adjacent to $v$ that are oriented towards $v$ in $O$. A multidegree $\md$ is called \emph{orientable}, if there is an orientation $O$ of the edges of $\GG_X$ such that $\md = \md_O$. Note that an orientation $O$ giving $\md$ need not be unique. An orientable multidegree has total degree \[\sum_{v \in \GG_X} \md_v = \sum_{v \in V(\GG_X)} \left( g_v - 1 + \indeg_O(v)\right) = |E(\GG_X)| - |V(\GG_X)| +   \sum_{v \in V(\GG_X)}  g_v = g - 1.\]

\tikzset{every picture/.style={line width=0.75pt}}
    
    \begin{figure}[ht]
    \begin{tikzpicture}[x=0.7pt,y=0.7pt,yscale=-0.8,xscale=0.8]
    \import{./}{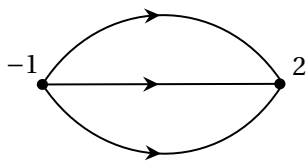}
    \end{tikzpicture}	
    \caption{
       An orientation and its associated multidegree. Both vertices are weightless and the three edges form a directed cut.
    } \label{figure1}
\end{figure}

\begin{lma} \label{lma:orientation_residual}
    Let $L$ be a line bundle on $X$ of multidegree $\md$. Assume $\md = \md_O$ and let $O'$ be the orientation obtained from $O$ by reversing the direction of each edge. Then \[\mdeg(\omega_X \otimes L^{-1}) = \md_{O'}.\]
\end{lma}

\begin{proof}
    For any vertex $v \in V(\GG_X)$, we have by definition $\md_v = g_v - 1 + \indeg_O(v)$. On the other hand, \[\mdeg(\omega_X \otimes L^{-1})_v = 2g_v - 2 + \val(v) - \md_v = g_v - 1 + \val(v) - \indeg_O(v).\]
    Any loop based at $v$ contributes $2$ to $\val(v)$ and $1$ to $\indeg_O(v)$. Any other edge adjacent to $v$ is either oriented towards $v$ or away from it. Thus $\val(v) - \indeg_O(v)$ equals the number of loop edges based at $v$ plus the number of edges adjacent to $v$ and oriented away from $v$ in $O$. This is by construction $\indeg_{O'}(v)$ and hence the claim follows.
\end{proof}

Recall that any subcurve $Y \subset X$ defines an induced subgraph $\GG_Y \subset \GG_X$.  The subset of edges of $\GG_X$ that are adjacent to a vertex in $\GG_Y$ and a vertex not in $\GG_Y$ is called the \emph{cut} defined by $\GG_Y$. They correspond to the nodes in $Y \cap Y^c$. If $O$ is an orientation on $\GG_X$, a cut is called a \emph{directed cut} if every edge in the cut is directed towards $\GG_Y$.

A \emph{cycle} in $\GG_X$ is a connected subgraph, in which each vertex has valence $2$. 
A \emph{directed cycle} in an orientation $O$ is a cycle in $\GG_X$, such that every vertex has exactly one edge directed towards it in the restriction of $O$ to the cycle; in particular, a loop edge gives a directed cycle in any orientation.

\begin{lma} \label{lma:cut cycle}
For every orientation $O$ on $\GG_X$, any edge $e$ of $\GG_X$ is contained in a directed cut or a directed cycle, but not both. 
\end{lma}

\begin{proof}
    This is well-known and usually attributed to \cite{Minty66}. We sketch a proof for the convenience of the reader. 
    Let $e$ be directed in $O$ from $v$ to $w$. 
    
    Suppose $e$ is contained in a directed cut and a directed cycle. Then the directed cut would restrict to a non-empty directed cut on the directed cycle, which is not possible.
    Thus $e$ is not contained in a directed cut and a directed cycle.
    
    To see that $e$ is contained in a directed cut or a directed cycle, let $V \subset V(\GG_X)$ denote the set of vertices $v'$ for which there exists a directed path from $w$ to $v'$ in $O$. If $v \in V$, $e$ completes the directed path from $w$ to $v$ to a directed cycle that contains $e$. If $v \not \in V$, the induced subgraph with vertices $V$ defines a directed cut containing $e$.
\end{proof}

\medskip

An orientation $O$ is called \emph{acyclic} if it contains no directed cycles, and hence by Lemma~\ref{lma:cut cycle} every edge is contained in a directed cut. 
It is called \emph{totally cyclic} if every edge is contained in a directed cycle, and hence by Lemma~\ref{lma:cut cycle} it contains no directed cuts.

A \emph{source} of an orientation $O$ is a vertex $v$ such that all edges adjacent to $v$ are oriented away from $v$. The following lemma is another well-known fact from graph theory.

\begin{lma}\label{lma:acyclic source}
Let $\GG_X$ be a graph and $O$ an orientation on $\GG_X$. 
\begin{enumerate}
\item The multidegree $\md_O$ is not effective, if and only if $O$ contains a source $v$ with $g_v = 0$.
\item If $O$ is an acyclic orientation, then it contains a source $v$.
\end{enumerate}   
\end{lma}

\begin{proof}
   The first claim follows immediately from the definition of $\md_O$.
   
   For the second claim observe that if $O$ is acyclic, then Lemma~\ref{lma:cut cycle} implies that $\GG_X$ contains a directed cut oriented towards an induced subgraph $\GG_Y$. The restriction of an acyclic orientation remains acyclic, and hence the restriction of $O$ to a connected component of $\GG_Y^c$ is acyclic. We can repeat this argument, until $\GG_Y^c$, the complement of $\GG_Y$, contains a single vertex, which by construction will be a source. 
\end{proof}

The following proposition motivates our interest in orientations in the context of semistable multidegrees.

\begin{prop}{\cite[Proposition 3.6]{Alexeev}} \label{prop:orientable}
Let $X$ be a stable curve and $\md$ a multidegree of total degree $g - 1$. Then $\md$ is semistable, if and only if $\md$ is orientable. It is stable, if and only if it can be given by a totally cyclic orientation.
\end{prop}

 \begin{rmk} \label{rmk:Hakimi}
 Inserting $d =  g-1$ in Definition~\ref{def:semistable}, a multidegree $\md$ of total degree $g - 1$ is semistable, if and only if for every induced subgraph $\GG_Y \subset \GG_X$, the total degree of $\md$ on $\GG_Y$ is at least $g(\GG_Y) - 1$. In this formulation, Proposition~\ref{prop:orientable} is a version of Hakimi's Theorem \cite{Ha65}. 
 \end{rmk}
 
\subsection{Existence of Theta divisors}\label{subsec:theta}

Apart from the combinatorial characterization of semistability in terms of orientations, there is the following characterization in terms of the existence of non-special line bundles.

\begin{prop}{\cite[Lemma 2.1]{Beauville}} \label{prop:beauville hzero}
Let $X$ be a stable curve and $\md$ a multidegree of total degree $g - 1$. Then $\md$ is semistable, if and only if there is a line bundle $L$ of multidegree $\md$ with $h^0(X, L) = 0$.
\end{prop}

Note that for $d= g-1$, the expected dimension of the effective locus $W_{\md}(X)$ is $g - 1$. Thus combining with Proposition~\ref{prop:determinantal}, it follows that if $\md$ is semistable, $W_{\md}(X)$ is either empty or a divisor in $\Pic^{\md}(X)$, the so called \emph{Theta divisor}. 
The next example shows, that the effective locus can indeed be empty.
\begin{ex}\label{ex:theta_empty}
Consider a curve $X$ with dual graph $\GG_X$ that has two vertices $v, w$ of weight $0$ and $k$ edges between $v$ and $w$. Let the multidegree $\md$ be given as $\md_v = -1$ and $\md_w = k - 1$ with total degree $g - 1 = k - 2$. Then $\md = \md_O$, where $O$ is given by orienting all edges from $v$ to $w$, and hence $\md$ is semistable by Proposition~\ref{prop:orientable}. See Figure~\ref{figure1} for the case $k = 3$. Let $L$ be any line bundle of multidegree $\md$ on $X$. Since $L$ has negative degree on $X_v$, any global section $s$ of $L$ vanishes along all of $X_v$. In particular, the restriction of $s$ to $X_w$ also vanishes at the $k$ points in $X_v \cap X_w$. Since $L|_{X_w}$ has degree $k - 1$, this is only possible for the zero section. Thus $h^0(X, L) = 0$ for all line bundles of multidegree $\md$ and $W_{\md}(X) = \emptyset$.   
\end{ex}

Recall that a multidegree $\md$ is called effective, if $\md_v \geq 0$ for all vertices $v$ of $\GG_X$.

\begin{thm} \label{thm:theta divisor}
Let $X$ be a stable curve and $\md$ a multidegree of total degree $g - 1$. Then we have for the effective locus $W_{\md}(X) \subset \Pic^{\md}(X):$
\begin{enumerate}
    \item $W_{\md}(X)$ is empty or has pure dimension $g - 1$ if and only if $\md$ is semistable. 
    \item $W_{\md}(X)$ is empty if and only if all irreducible components of $X$ are rational and $\md = \md_O$ with $O$ acyclic. If $\md$ is effective, $W_{\md}(X)$ is not empty.
\end{enumerate}
\end{thm}

\begin{proof}
    As mentioned above, the first claim follows by Propositions~\ref{prop:determinantal} and \ref{prop:beauville hzero}.
    For the second claim, we show more generally that for an orientable multidegree $\md = \md_O$ on a nodal curve $X$, which is not necessarily stable or connected, $W_{\md}(X) = \emptyset$ if and only if $g_v = 0$ for all $v \in V(\GG_X)$ and $O$ is acyclic. Note first, that to prove this stronger claim, we still may assume that $X$ is connected, since the condition can be checked on each connected component.
    
    We first prove the last claim: if $\md$ is effective, then $W_{\md}(X) \neq \emptyset$. Indeed, in this case let $p_i$ be a collection of smooth points of $X$, $\md_v$ of them contained in $X_v$. Then $\cO_{X_v}(\sum_{i = 1}^{d} p_i)$ is effective of multidegree $\md$. Furthermore, if $\md$ is effective and orientable, either $g_v \neq 0$ for some $v$ or $O$ is not acyclic by Lemma~\ref{lma:acyclic source}. Thus we can assume from now on that $\md$ is not effective.
    
    We show the claim by induction on the number of non-loop edges of $\GG_X$. The base case is if $\GG_X$ contains no non-loop edges, that is, $X$ is irreducible and $\GG_X$ contains a single vertex $v$. Since $\md$ is orientable and not effective, we need to have $g_v = 0$ and $\GG_X$ contains no edges, in which case $\md_v = -1$ and $W_{\md}(X) = \emptyset$ is immediate.
    
    For the induction step, let $v$ be a vertex such that $\md_v < 0$. 
    By Lemma~\ref{lma:acyclic source} (1), we need to have $g_v = 0$ and $v$ is a source in $O$. Let $X_v^c$ be the closure of the complement of $X_v$ in $X$. For a line bundle $L$ on $X$ of multidegree $\md$, let $\md'$ denote the multidegree of $L|_{X_v^c}(-X_v \cap X_v^c)$. Since all vertices are oriented away from $v$ in $O$, $\md'$ is orientable. Namely, it is given by the restriction of $O$ to the dual graph of $X_v^c$. Since any global section of $L$ vanishes along all of $X_v$, we have \begin{equation}\label{eq:theta1}
        h^0(X,L) = h^0\left(X_v^c, L|_{X_v^c}(-X_v \cap X_v^c)\right).
    \end{equation}
    Indeed, $H^0\left(X_v^c, L|_{X_v^c}(-X_v \cap X_v^c)\right)$ is naturally identified with the kernel of the evaluation map $H^0(X,L) \to H^0(X_v, L|_{X_v})$. If all global sections of $L$ vanish along $X_v$, then $H^0(X,L)$ coincides with this kernel and \eqref{eq:theta1} follows. 
    
    
    Now suppose, that there is a line bundle $L$ on $X$ with multidegree $\md$ and $h^0(X,L) \geq 1$. Then also 
    $h^0\left(X_v^c, L|_{X_v^c}(-X_v \cap X_v^c)\right) \geq 0$ by \eqref{eq:theta1}. By induction it follows that there is either $w \in V(\GG_X) \setminus \{v\}$ with $g_w \neq 0$ or $\md' = \md_{O'}$ with $O'$ not acyclic. In the first case, $X_w$ is a non-rational component also of $X$. In the second case, we may extend $O'$ to an orientation $O$ on $\GG_X$ by orienting the edges adjacent to $v$ away from $v$. By construction, we have $\md = \md_O$. If $C$ is an oriented cycle of $O'$, then it is also an oriented cycle of $O$ and hence $O$ is not acyclic.

    Conversely, suppose that $W_{\md}(X)$ is empty. Then also $W_{\md'}(X_v^c)$ is empty by \eqref{eq:theta1}. So by induction, $g_w = 0$ for all $w \in V(\GG_X) \setminus \{v\}$ and $\md' = \md_{O'}$ with $O'$ acyclic. By assumption, we have also $g_v = 0$, and hence all irreducible components of $X$ are rational. We can extend $O'$ to an orientation $O$ with $\md = \md_O$ as above by orienting the edges adjacent to $v$ away from $v$. We are finished if we can show that $O$ is acyclic. To this end, suppose to the contrary that $C$ is an oriented cycle of $O$. By Lemma~\ref{lma:cut cycle}, $C$ contains none of the edges adjacent to $v$ since they are all contained in a directed cut. Thus $C$ would restrict to an oriented cycle of $O'$, which is not possible.
\end{proof}

\begin{rmk} \label{rmk:correction old statement}
Example~\ref{ex:theta_empty} contradicts \cite[Proposition 2.2]{Beauville}, in which it is claimed that for total degree $d = g - 1$, the effective locus is a divisor in $\Pic^{\md}(X)$ if and only if $\md$ is semistable. The possibility of an empty effective locus is not accounted for in the proof of \emph{loc. cit.}, and Theorem~\ref{thm:theta divisor} provides the necessary correction.

In \cite{Beauville} Theta divisors are used to define a polarization on generalized Prym varieties. The multidegree $\md$ giving the Theta divisor is that of a Theta characteristic $L_0$, that is, a line bundle such that $L_0^{\otimes 2}$ is isomorphic to the dualizing sheaf of $X$. Thus the multidegree $\md$ of a Theta characteristic has value $\md_v = g(X_v) - 1 + \frac{|X_v \cap X_v^c|}{2}$ on $X_v$. In particular, $\md$ is effective whenever $X$ is not a disjoint union of smooth rational curves. The subcurves used in \cite[Proposition 5.2 and Theorem 5.4]{Beauville} to define the generalized Prym varieties are not the disjoint union of smooth rational curves. Thus no issue arises, since a Theta divisor exists if $\md$ is effective and semistable by Theorem~\ref{thm:theta divisor}. 

The second main application of Theta divisors concerns the theory of compactified Jacobians. Recall that the compactified Jacobian  admits a stratification with strata $P^{(\md, S)}$, parametrizing line bundles that have stable multidegree $\md$ on the partial normalization of $X$ at a subset of nodes $S$. By Proposition~\ref{prop:orientable}, a stable multidegree $\md$ in degree $g  - 1$ is given by a totally cyclic orientation on the dual graph of this partial normalization, i.e., a subgraph of $\GG_X$. Since a totally cyclic orientation is acyclic only if the graph has no edges, Theorem~\ref{thm:theta divisor} gives that the only stratum in which $W_{\md}(X)$ is not a divisor is the one in which $S$ is the set of all nodes and then also only if $g_v = 0$ for all irreducible components $X_v$ of $X$. But in this case $P^{(\md, S)}$ has dimension $0$ so no issue arises. In particular, the use of \cite[Proposition 2.2]{Beauville} as \cite[Lemma 3.8]{Alexeev} and \cite[Proposition 1.3.7]{caporasotheta} to extend the Theta divisor to compactified Jacobians causes no issues.
\end{rmk}

\subsection{Rational Abel maps and components of Theta divisors}\label{subsec:abel}

We next recall a description of the irreducible components of a Theta divisor $W_{\md}(X)$ given in \cite{coelhoesteves}.

For an irreducible component $X_v$ of $X$, denote by $X_v^{\mathrm{sm}}$ the locus of points in $X_v$ that are smooth points of $X$. Following \cite[Section 1.2.7]{caporasotheta}, we define the \emph{rational Abel map} associated to an effective multidegree $\md$ of total degree $d$ as the map
\[
\alpha_{\md}\colon \prod_{v \in V(\GG_X)} \left ( X_v^{\mathrm{sm}} \right)^{\md_v} \to \Pic^{\md}(X),
\]
given by sending a $d$-tuple of points $(p_1, \dots, p_d)$ to $\mathcal O_X\left(\sum_{i = 1}^d p_i\right)$. We denote the image of $\alpha_{\md}$ by $A_{\md}(X)$; clearly, it is irreducible of dimension at most $\min\{d, g\}$. 

\begin{lma} \label{lma:dimension Abel map}
Let $\md$ be an effective multidegree of total degree $0 \leq d \leq g$. Then the dimension of $A_{\md}(X)$ is $d$ if and only if a general $[L] \in A_{\md}(X)$ satisfies $h^0(X, L) = 1$.  
\end{lma}
\begin{proof}
    The domain of $\alpha_{\md}$ is irreducible of dimension $d$, thus $A_{\md}(X)$ has dimension $d$ if and only if $\alpha_{\md}$ is generically finite. 
    
    The fiber of $\alpha_{\md}$ over a point $[L] = \left [\mathcal O_X\left(\sum_{i = 1}^d p_i\right)\right]$ in the image of $\alpha_{\md}$ consists of tuples $(p_1', \dots, p_d')$ such that $\mathcal O_X\left(\sum_{i = 1}^d p_i \right) \simeq \mathcal O_X\left(\sum_{i = 1}^d p_i' \right)$. In particular, if $h^0(X, L) = 1$ for a general $L$ in $A_{\md}(X)$, then $\alpha_{\md}$ is generically finite. 
    
    Conversely, suppose a line bundle of the form $L = \mathcal O_X\left(\sum_{i = 1}^d p_i \right)$ admits two linearly independent global sections $s_1, s_2$, the first of them vanishing only at the $p_i$. Then the subspace of $H^0(X,L)$ spanned by the $s_i$ consists of sections of the form $a s_1 + b s_2$ for $a,b \in k$. For general choices of $a,b$, these global sections do not vanish along nodes or entire irreducible components of $X$, since this is true for $s_1$. Hence their associated divisor is of the form $p_1' + \dots + p_d'$ for smooth points $p_i'$ and $\alpha_{\md}$ has fibers of dimension at least $1$ over such points $[L]$. Thus if the general $[L] \in A_{\md}$ satisfies $h^0(X,L) \geq 2$, then $\alpha_{\md}$ is not generically finite, as claimed.   \end{proof}

\begin{rmk}\label{rmk:abel}
The assumption of Lemma~\ref{lma:dimension Abel map}, that a general $[L] \in A_{\md}(X)$ satisfies $h^0(X, L) = 1$, is not automatically satisfied, even if $d \leq g$. For example, suppose $X$ contains a smooth rational component $X_v$, and $\md$ is the multidegree with value $\md_v = |X_v \cap X_v^c| + 1$ on $v$ and $0$ on all other irreducible components of $X$. Then every line bundle $L$ in $A_{\md}(X)$ satisfies $h^0(X,L) \geq 2$ even though $\md$ is effective. In this case $\dim\left(A_{\md}(X)\right) < d$ and it follows from Proposition~\ref{prop:determinantal}, that $A_{\md}(X)$ is not an irreducible component of the effective locus $W_{\md}(X)$.
\end{rmk}

Let $L$ be a line bundle of multidegree $\md$ and $Y \subset X$ a subcurve. Let $\nu \colon X^\nu \to X$ be the partial normalization of $X$ at nodes in $Y \cap Y^c$. We obtain the pull-back map \[\nu^* \colon \Pic^{\md}(X) \to \Pic^{\md_Y}(Y) \times \Pic^{\md_{Y^c}}(Y^c),\] where $\md_{Y}$ and $\md_{Y^c}$ are the multidegrees of $L|_Y$ and $L|_{Y^c}$, respectively. Suppose now, that the multidegree $\md'$ of $L|_{Y}(-Y \cap Y^c)$ is effective on $Y$. Then we can define
\[
V \coloneqq \left\{\left[L'(Y \cap Y^c)\right] \mid \, [L'] \in A_{\md'}(Y)\right \} \subset \Pic^{\md_Y}(Y)
\]
and
\[
W_{\md, Y}(X) \coloneqq (\nu^*)^{-1}\left(\Pic^{\md_{Y^c}}(Y^c) \times V\right) \subset \Pic^{\md}(X).
\]
Thus, roughly speaking, $W_{\md, Y}(X)$ is the locus of line bundles on $X$ of multidegree $\md$, that are in the image of a rational Abel map on $Y$ shifted by $Y \cap Y^c$, and arbitrary away from $Y$. 

Recall that by Proposition~\ref{prop:orientable} a multidegree $\md$ of total degree $g - 1$ is semistable if and only if it is orientable, i.e., $\md = \md_O$ for some orientation $O$ on $\GG_X$. 
\begin{prop}{\cite[Theorem 3.6]{coelhoesteves}} \label{prop:coelhoesteves}
Let $X$ be a stable curve and $\md = \md_O$ a semistable multidegree of total degree $g - 1$. Then 
\[
W_{\md}(X) = \bigcup_Y W_{\md, Y}(X),
\]
and the $W_{\md, Y}(X)$ are irreducible divisors in $\Pic^{\md}(X)$. The union is over connected subcurves $Y \subset X$, such that the multidegree associated to the restriction $O|_{\GG_Y}$ of $O$ is effective on the dual graph $\GG_Y \subset \GG_X$, and $\GG_Y$ defines a directed cut with edges oriented towards $\GG_Y$ in $O$.
\end{prop}

\begin{rmk}
In case that $\md = \md_O$ is a stable multidegree, $O$ contains by Proposition ~\ref{prop:orientable} no directed cut except for the empty one given by $Y = X$. In this case, Proposition~\ref{prop:coelhoesteves} implies that $W_{\md}(X)$ is an irreducible divisor, as was established in \cite[Theorem 3.1.2]{caporasotheta}. Notice also, that Theorem~\ref{thm:theta divisor} is consistent with Proposition~\ref{prop:coelhoesteves} concerning emptiness of the effective locus: one can argue similar as in the proof of Theorem~\ref{thm:theta divisor}, to show that there exists no connected subcurve $Y$ of $X$ as required in Proposition~\ref{prop:coelhoesteves} if and only if $X$ has only rational components and $O$ is acyclic.
\end{rmk}

\begin{lma} \label{lma:coelhoesteves}
Let $W_{\md, Y}(X)$ be an irreducible component of $W_{\md}(X)$ as in Proposition~\ref{prop:coelhoesteves}. Let $p \in X_v$ be a smooth point with $X_v \subset Y$. Then a general line bundle $[L] \in W_{\md, Y}(X)$ satisfies $h^0(X,L) = 1$ and $p$ is not a base point of $L$.  
\end{lma}

\begin{proof}
 The first claim, that $h^0(X,L) = 1$ for $L$ general in $W_{\md, Y}(X)$, is part of \cite[Proposition 3.5]{coelhoesteves}. The second claim follows, since global sections of $L$ vanishing on $Y^c$ can by construction of $W_{\md, Y}(X)$ be identified with those of a sheaf of the form $\mathcal O_Y\left(\sum_{i = 1}^{k} p_i \right)$ for some smooth points $p_i$. Thus for $L$ such that $p \neq p_i$ for all $i$, $p$ is not a base point of $L$.
\end{proof}

\section{The cases \texorpdfstring{$d = g - 2$}{d = g - 2} and \texorpdfstring{$d = g$}{d = g}}\label{sec:g-2}

In this section, we study effective loci for semistable multidegrees of total degree $d = g - 2$, and its residual case $d = g$. The main results are Proposition~\ref{prop:dimension deg g - 2} and Theorem~\ref{thm:char deg g}.

\subsection{Dimension of special loci} \label{subsec:dimension g}

In \cite{CC}, an analogue in degree $g$ of the description in terms of orientations in degree $g - 1$ of Proposition~\ref{prop:orientable} was established.
Recall that for an orientation $O$ an induced subgraph $\GG_Y \subset \GG_X$ induces a directed cut, if all edges of $\GG_X$ adjacent to $\GG_Y$ but not contained in $\GG_Y$ are directed towards $\GG_Y$ in $O$. 

\begin{lma}{\cite[Lemmas 1.7.4 and 3.3.2]{CC}}\label{lma:orientable g}
A multidegree $\md$  of total degree $g$ is semistable
if and only if there is $v \in V(\GG_X)$ and an orientation $O$ on $\GG_X$, such that $\md = \md_O + v$ and there is no directed cut oriented towards an induced subgraph $\GG_Y$ with $v \in V(\GG_Y)$. If $\md$ is semistable, such an orientation exists for every choice of vertex $v \in V(\GG_X)$.
\end{lma}

\begin{rmk}\label{rmk:partial orientation}
Recall that if $\mdeg(L) = \md_O$ is an orientable multidegree, then $\mdeg(\omega_X \otimes L^{-1}) = \md_{O'}$, where $O'$ is the orientation obtained from $O$ by reversing the direction of all edges (see Lemma~\ref{lma:orientation_residual}). Since semistability is preserved in passing to the residual by Lemma~\ref{lma:semistable residual},  Lemma~\ref{lma:orientable g} thus gives the following characterization in total degree $d = g-2$. A multidegree $\md$ of total degree $d = g- 2$ is semistable if and only if for every vertex $v \in V(\GG_X)$, we have $\md = \md_O - v$ with $O$ an orientation on $\GG_X$ such that there is no directed cut oriented away from an induced subgraph $\GG_Y$ with $v \in V(\GG_Y)$. If $\md$ is semistable, such an orientation exists for every choice of vertex $v \in V(\GG_X)$.
\end{rmk}

The characterizations of semistability in Lemma~\ref{lma:orientable g} and Remark~\ref{rmk:partial orientation} can be encoded conveniently by biorienting one edge, respectively leaving one edge unoriented in the orientation. The condition then is that all directed cuts are oriented away from the bioriented edge, respectively towards the unoriented edge. See \cite{CC} for examples in degree $g$ and Figure~\ref{figure2} for an illustration in degree $g-2$.

\tikzset{every picture/.style={line width=0.75pt}}
    
    \begin{figure}[ht]
    \begin{tikzpicture}[x=0.7pt,y=0.7pt,yscale=-0.8,xscale=0.8]
    \import{./}{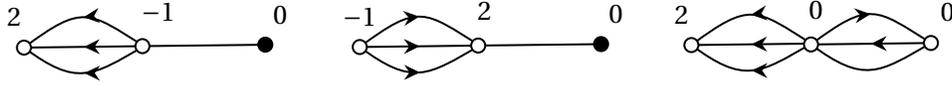}
    \end{tikzpicture}	
    \caption{
     Multidegrees given by partial orientations as in Remark~\ref{rmk:partial orientation}, where bold vertices have weight $1$ and circled ones weight $0$. Only the multidegree in the middle picture is semistable.
    } \label{figure2}
\end{figure}

\begin{lma} \label{lma: g-2 to g-1}
Let $X$ be a stable curve and $\md$ a multidegree of total degree $d$.
\begin{enumerate}
    \item If $d = g - 2$, $\md$ is semistable if and only if $\md + v$ is semistable for all vertices $v \in V(\GG_X)$.
    \item If $d = g$, $\md$ is semistable if and only if $\md - v$ is semistable for all vertices $v \in V(\GG_X)$.
\end{enumerate}
\end{lma}

\begin{proof}
    We show the second claim and assume $d = g$. The first statement for $d = g - 2$ follows by Lemma~\ref{lma:semistable residual} from the second statement via passing to residuals. 
    So suppose first, that $\md$ is semistable of total degree $g$. Then Lemma~\ref{lma:orientable g} implies that $\md - v$ is orientable for every  vertex $v \in V(\GG_X)$, and hence  $\md - v$ is semistable by Proposition~\ref{prop:orientable}.
    
    Conversely, suppose $\md - v$ is semistable for every vertex $v \in V(\GG_X)$. Again by Proposition~\ref{prop:orientable}, it follows that $\md_O = \md - v$ for some orientation $O$. To conclude that $\md$ is semistable, it suffices by Lemma~\ref{lma:orientable g} to show, that for every directed cut that is oriented towards an induced subgraph $\GG_Y$ we have $v \not \in V(\GG_Y)$. 
    
    Suppose to the contrary, that $v \in V(\GG_Y)$ and $\GG_Y$ induces a directed cut oriented towards $\GG_Y$. Let $\GG_Y^c$ be the dual graph of $Y^c = \overline{X \setminus Y}$. Since by assumption all edges that are directed towards a vertex in $\GG_Y^c$ are already contained in $\GG_Y^c$, the restriction of $\md$ to $\GG_Y^c$ is induced by the restriction of $O$ to $\GG_Y^c$. In particular, we have $\sum_{X_w \subset Y^c} \md_w = g(\GG_Y^c) - 1$. Now fix a $w' \in V(\GG_Y^c)$ and consider  $\md - w'$. Then \[\sum_{X_w \subset Y^c} (\md- w')_w = g(\GG_Y^c) - 2 < g(\GG_Y^c) - 1\] and thus $\md - w'$ is not semistable by Remark~\ref{rmk:Hakimi}. This contradicts the assumption.
\end{proof}

\begin{prop}\label{prop:dimension deg g - 2}
Let $X$ be a stable curve and $\md$ a semistable multidegree of total degree $d$. 
\begin{enumerate}
    \item If $d = g-2$, the effective locus $W_{\md}(X)$ is either empty, or of pure dimension $g - 2$.
    \item If $d = g$, the special locus $W^1_{\md}(X)$ is either empty, or of pure dimension $g - 2$.
\end{enumerate}
\end{prop}

\begin{proof}
     Suppose first $d = g -2$.
     If $W_{\md}(X) = \emptyset$, there is nothing to show. Otherwise, let $W \subset W_{\md}(X)$ be an irreducible component. Let $v \in V(\GG_X)$ be a vertex of the dual graph. By Proposition~\ref{prop:orientable} and Lemma~\ref{lma: g-2 to g-1}, we have $\md + v = \md_O$ for some orientation $O$ of $\GG_X$. Choose a smooth point $p \in X_v$ and consider the isomorphism 
    \[
    \varphi_p \colon \Pic^{\md}(X) \to \Pic^{\md_O}(X), [L] \mapsto [L(p)].
    \]
    If $L$ is effective, then so is $L(p)$. Thus the image $\varphi_p(W)$ is contained in an irreducible component of the effective locus in $\Pic^{\md_O}(X)$. Recall that Proposition~\ref{prop:coelhoesteves} gives a description of such irreducible components $W_{\md_O, Y}(X)$. Here $Y$ is a subcurve such that $\GG_Y$ induces a directed cut  in $O$, oriented towards $\GG_Y$. So we may assume $\varphi_p(W) \subset W_{\md_O, Y}(X)$. 
    
    By Remark~\ref{rmk:partial orientation}, we need to have $X_v \subset Y$ since the cut induced by $\GG_Y$ is directed towards $\GG_Y$ in $O$ and $\md$ is semistable. Thus we can apply Lemma~\ref{lma:coelhoesteves}, and obtain that
    for a general line bundle $[L'] \in W_{\md_O, Y}(X)$, we have $h^0(X, L') = 1$ and $p$ is not a base point of $L'$. This in turn implies that
    $\varphi_p(W) \neq W_{\md_O, Y}(X)$ since either $p$ is a base point of $L(p)$ or $h^0(X,L(p)) = h^0(X,L) + 1 \geq 2$. By Theorem~\ref{thm:theta divisor}, we have $\dim\left(W_{\md_O, Y}(X)\right) = g - 1$ since $\md_O$ is semistable. Since $\varphi_p(W)$ is a closed subset of  $W_{\md_O, Y}(X)$, this implies $\dim(W) \leq g - 2$. The claim now follows from Proposition~\ref{prop:determinantal}, since the expected dimension of $W$ is $d = g - 2$.

    For $d = g$ the claim follows from the case of $d = g -2$ by passing to residuals (see Remark~\ref{rmk:residual}).
\end{proof}

\subsection{Characterizing semistability} \label{subsec:semistability g}

Recall that Theorem~\ref{thm:theta divisor} (1) states that in degree $g - 1$ the effective locus is empty or of expected dimension if and only if $\md$ is semistable. Thus Proposition~\ref{prop:dimension deg g - 2} generalizes the `if' part of this statement to degree $g - 2$. The `only if' part no longer holds, as the following examples show.  

\begin{ex}\label{ex:g-2 1}
Let $\GG_X$ be the graph with three vertices $v_1, v_2, v_3$ of weight $(0, 0, 1)$, a triple edge between $v_1$ and $v_2$, and a single edge between $v_2$ and $v_3$. Consider the multidegree $\md = (2, -1, 0)$. Its total degree is $1 = g - 2$ and one checks that $\md$ is not semistable. See the left picture in Figure~\ref{figure2} for an illustration. However, the effective locus $W_{\md}(X)$ is empty in this case. Notice also, that the multidegree $(-1, 2, 0)$ on the same graph, as in the middle picture in Figure~\ref{figure2}, is semistable with empty effective locus. 
\end{ex}

\begin{ex}\label{ex:g-2 2}
Let $\GG_X$ be the graph with vertices $v_1, v_2, v_3$, all of weight $0$, and three edges between each of the pairs $v_1, v_2$ and $v_2, v_3$. Denote by $X_i$ the irreducible component of $X$ corresponding to $v_i$. Consider the multidegree $\md = (2, 0, 0)$. Its total degree is $2 = g - 2$ and one checks that $\md$ is not semistable. See the right picture in Figure~\ref{figure2} for an illustration. Effective line bundles $[L] \in \Pic^{\md}(X)$ are exactly those with $L|_{X_2 \cup X_3} = \mathcal O_{X_2 \cup X_3}$. Thus $W_{\md}(X)$ is irreducible and $2$-dimensional, the two parameters being given by the gluing data along the three nodes in $X_1 \cap X_2$. 
\end{ex}

Note that for the dual graph in Example~\ref{ex:g-2 1}, the effective locus for a semistable multidegree of total degree $g -1$ is never empty by Theorem~\ref{thm:theta divisor} (2) since there is $v$ with $g_v \geq 1$, whereas it is empty for the semistable multidegree $(-1, 2, 0)$ of total degree $g - 2$. In general, we do not know, how to generalize the second part of Theorem~\ref{thm:theta divisor} to degree $g-2$, that is: 

\begin{Question}
When is the effective locus for a semistable multidegree of total degree $g - 2$ empty? 
\end{Question}

Examples~\ref{ex:g-2 1} and \ref{ex:g-2 2} show in particular, that semistability in degree $g-2$ cannot be characterized by the existence of non-special line bundles as in Proposition~\ref{prop:beauville hzero} for degree $g-1$. We next give a refinement that does allow for a similar characterization.

\begin{lma}\label{lma:char deg g-2}
Let $X$ be a stable curve and $\md$ a multidegree of total degree $g - 2$. Then $\md$ is semistable if and only if
there is a line bundle $L$ of multidegree $\md$ such that $h^0(X,L) = 0$ and for every irreducible component $X_v$ of $X$ and a general point $p_v \in X_v$, we still have $h^0\left(X,L(p_v)\right) = 0$.
\end{lma}

\begin{proof}
    If there is a line bundle $L$ of multidegree $\md$ as described in the lemma, pick a smooth point $p_v \in X_v$ for every irreducible component $X_v$ such that $h^0(X, L(p_v)) = 0$. Thus $\md + v$ is semistable for every $v \in V(\GG_X)$ by Proposition~\ref{prop:beauville hzero} and $\md$ is semistable by Lemma~\ref{lma: g-2 to g-1}. 
    
    Conversely, suppose $\md$ is semistable. Then $\md + v$ is semistable for all vertices $v \in V(\GG_X)$ by Lemma~\ref{lma: g-2 to g-1}. Choose a smooth point $p_v \in X_v$ on each irreducible component $X_v$. This gives the isomorphism \[\varphi_{p_v} \colon \Pic^{\md}(X) \to \Pic^{\md + v}(X), [L] \mapsto [L(p_v)].\] Under this identification, there is a dense open set  $U_v \subset \Pic^{\md}(X)$ such that for all line bundles $[L] \in U_v$ we have $h^0(X,L(p_v)) = 0$ by Theorem~\ref{thm:theta divisor}, since $\md + v$ is semistable. Then an $[L]$ in the intersection of the finitely many $U_v$ satisfies the conditions in the lemma by construction.
\end{proof}

The residual statement of Lemma~\ref{lma:char deg g-2} gives the statement of Theorem~\ref{thm:dominant abel}. To prove it, we will need the following easy consequence of the Riemann Roch theorem:
\begin{lma}\label{lma:RR base point}
Let $X$ be a stable curve, $p \in X$ a smooth point and $L$ a line bundle on $X$. Then $p$ is a base point of $L$ if and only if it is not a base point of $\left(\omega_X \otimes L^{-1}\right)(p)$. 
\end{lma}

\begin{proof}
Applying the Riemann Roch theorem gives on the one hand \[h^0\left(X,L\right) - h^0\left(X, \omega_{X} \otimes L^{-1}\right) = d - g + 1\] and on the other \[h^0\left(X,L(-p)\right) - h^0\left(X, (\omega_{X} \otimes L^{-1})(p)\right) = d - 1 - g + 1.\]
Hence $h^0(X, L) = h^0\left(X, L(-p)\right)$ if and only if $h^0\left(X, \omega_{X} \otimes L^{-1}\right) = h^0\left(X, (\omega_{X} \otimes L^{-1})(p)\right) - 1$.
\end{proof}

Recall that we defined the rational Abel map in Subsection~\ref{subsec:abel}.

\begin{thm} \label{thm:char deg g}
Let $X$ be a stable curve and $\md$ a multidegree of total degree $g$. Then the following are equivalent:
\begin{enumerate}
    \item The multidegree $\md$ is semistable. 
    \item There is a line bundle $L$ of multidegree $\md$ such that $h^0(X,L) = 1$ and non-zero global sections of $L$ do not vanish along a whole irreducible component $X_v$ of $X$.
    \item The multidegree $\md$ is effective and the associated rational Abel map $\alpha_{\md}$ is dominant. 
\end{enumerate}
\end{thm}

\begin{proof}
    By Lemma~\ref{lma:semistable residual} a line bundle $L$ is semistable if and only if its residual $\omega_X \otimes L^{-1}$ is so. If $L$ has total degree $g$, then   $\omega_X \otimes L^{-1}$ has total degree $g - 2$. Thus, using Lemma~\ref{lma:char deg g-2}, $\md$ is semistable if and only if there is a line bundle $L$ of multidegree $\md$, such that $h^0\left(X, \omega_X \otimes L^{-1}\right) = 0$ and for any irreducible component $X_v$ of $X$ and general point $p_v \in X_v$, $p_v$ is a base point of $\left(\omega_X \otimes L^{-1}\right)(p_v)$. By the Riemann Roch theorem and Lemma~\ref{lma:RR base point}, these two conditions translate to condition (2) on $L$. Thus equivalence of (1) and (2) follows. 
    
    Next, we show that (3) implies (2). If the image $A_{\md}(X)$ of the rational Abel map $\alpha_{\md}$ has dimension $g$, then a general $[L]$ in $A_{\md}(X)$ satisfies $h^0(X,L) = 1$ by Lemma~\ref{lma:dimension Abel map}. By definition, such an $L$ is of the form $\cO_X\left(\sum_{i=1}^d p_i\right)$, which satisfies also the second part of condition (2), since the global sections of $\cO_X$ do not vanish along whole irreducible components of $X$.
    
    Finally, suppose that $\md$ satisfies the equivalent conditions (1) and (2). If $\deg\left(L|_{X_v}\right) < 0$ for some irreducible component $X_v$, any global section of $L$ vanishes along all of $X_v$. Thus (2) immediately implies that $\md$ is effective and thus the rational Abel map $\alpha_{\md}$ is defined. By Lemma~\ref{lma:dimension Abel map}, if a general $[L] \in A_{\md}(X)$ satisfies $h^0(X,L) = 1$, then $\dim\left(A_{\md}(X)\right) = d = g$. Since this is also the dimension of $\Pic^{\md}(X)$, it suffices to show this in order to prove (3). 
    
    Since $d > 0$ and $\md$ is effective, there exists a vertex $v$, such that $\md - v$ is still effective. Hence we can consider the corresponding rational Abel map and its image $A_{\md - v}(X) \subset \Pic^{\md - v}(X)$. Fix a smooth point $p_v \in X_v$ and consider once again the isomorphism
    \[\varphi_{p_v} \colon \Pic^{\md - v}(X) \to \Pic^{\md}(X), [L] \mapsto [L(p_v)].\] Then by construction $\varphi_{p_v}\left(A_{\md - v}(X)\right) \subset A_{\md}(X)$. Since $\md - v$ is semistable by Lemma~\ref{lma: g-2 to g-1}, it follows that $A_{\md - v}(X)$ is an irreducible component of the effective locus in $\Pic^{\md - v}(X)$ by Proposition~\ref{prop:coelhoesteves} and has dimension $g - 1$ by Theorem~\ref{thm:theta divisor}. Since $\varphi_{p_v}$ is an isomorphism, we get that $A_{\md}(X)$ has dimension at least $g - 1$. In particular, $A_{\md}(X)$ cannot be contained in the special locus, which by Proposition~\ref{prop:dimension deg g - 2} is either empty or has dimension $g - 2$. In other words, a general $[L] \in A_{\md}(X)$ satisfies $h^0(X,L) = 1$, as claimed. 
\end{proof}

\section{Effective loci for general semistable multidegrees}\label{sec:general}

Finally, we consider in this section the dimension of effective loci for semistable multidegrees of total degree $d \leq g - 2$. In Subsection~\ref{subsec:combinatorics} we lay some combinatorial groundwork, and in Subsection~\ref{subsec: general effective loci} we prove Theorem~\ref{thm:general effective loci}, the main result of this section. We conclude with some counterexamples to possible stronger claims in Subsection~\ref{subsec:further counterexamples}.

\subsection{Combinatorial considerations} \label{subsec:combinatorics}
Recall that a line bundle $L$ of total degree $d$ is by definition semistable if and only if \[g(Y) - 1 + (d - g + 1) \frac{2 g(Y) - 2 + |Y \cap Y^c|}{2g - 2} \leq \deg(L|_Y),\] for every subcurve $Y$ of $X$. Recall furthermore, that if $d = g$ and $X$ is stable, this is equivalent to requiring \[g(Y) \leq \deg(L|_Y)\] for every subcurve $Y$ of $X$, since $0 < \frac{2 g(Y) - 2 + |Y \cap Y^c|}{2g - 2} < 1$ for proper subcurves $Y$.
 
\begin{lma} \label{lma:vertex to subtract from}
Let $X$ be a stable curve and $L$ a line bundle of total degree $d > g$. Suppose $g(Y) \leq \deg(L|_Y)$ for all subcurves $Y \subset X$. Then there is a vertex $v \in V(\GG_X)$ such that for all subcurves $Y$ with $X_v \subset Y$ the strict inequality $g(Y) < \deg(L|_Y)$ holds.
\end{lma}

\begin{proof}
    We first claim, that if $L$ is as in the assumptions of the lemma and $Y$ is a subcurve such that $\deg\left(L|_Y\right) = g(Y)$, then $Y$ needs to be connected. Indeed, if to the contrary $Y = Y_1 \sqcup Y_2$, then $g(Y) = g(Y_1) + g(Y_2) - 1$. Since $\deg\left(L|_Y\right) = \deg\left(L|_{Y_1}\right) + \deg\left(L|_{Y_2}\right)$ we thus need to have $\deg\left(L|_{Y_1}\right) < g(Y_1)$ or $\deg\left(L|_{Y_2}\right) < g(Y_2)$. But this contradicts the assumption on $L$.
    
    Now let $Y, Z \subset X$ be two connected subcurves such that $\deg\left(L|_Y\right) = g(Y)$ and $\deg\left(L|_Z\right) = g(Z)$. Assume furthermore, that $Y \cup Z$ is connected, as well. Notice that both $Y$ and $Z$ need to be proper subcurves of $X$, since $d > g$.
     We claim, that we then need to also have $\deg\left(L|_{Y \cup Z}\right) = g(Y \cup Z)$. 
     By assumption, we have $\deg\left(L|_{Y \cup Z}\right) \geq g(Y \cup Z)$ and we need to show the other inequality.
     
    On the one hand, we have
    \begin{equation} \label{eq:degree union}
       \deg\left(L|_{Y \cup Z}\right) = g(Y) + g(Z) - \deg\left(L|_{Y \cap Z}\right),
    \end{equation}
    where we set $\deg\left(L|_{Y \cap Z}\right) = 0$ if $Y \cap Z$ consists only of nodes of $X$.
    Write $k$ for the number of edges in $\GG_X$ that are adjacent to both dual graphs $\GG_Y$ and $\GG_Z$ of $Y$ and $Z$, but contained in neither. Then one checks using the definition of genus \eqref{eq:def genus}, that we have on the other hand
    \begin{equation}\label{eq: adding genera}
        g(Y) + g(Z) = g(Y \cup Z) + 1 - k +  g(Y \cap Z) - l,
    \end{equation}
    where $l = 0$ if $Y \cap Z$ consists only of nodes, in which case we set $g(Y \cap Z) = 0$, and $l = 1$ otherwise. 
    Substituting \eqref{eq: adding genera} in \eqref{eq:degree union} gives
    \begin{eqnarray*}
     \deg\left(L|_{Y \cup Z}\right) &=&  g(Y \cup Z) + 1 - k + g(Y \cap Z) - l - \deg\left(L|_{Y \cap Z}\right)\\
     &\leq& g(Y \cup Z) + 1 - k - l,
    \end{eqnarray*}
    where we used that by assumption $\deg\left(L|_{Y \cap Z}\right) \geq g(Y \cap Z)$. Since $Y \cup Z$ is connected, not both $k$ and $l$ can be $0$, and hence  $\deg\left(L|_{Y \cup Z}\right)
     \leq g(Y \cup Z)$, as claimed.
    
    To conclude assume to the contrary, that for every vertex $v$ there is a subcurve $Y \subset X$ containing $X_v$ and such that $\deg\left(L|_{Y}\right) = g(Y)$. Then the union of all such subcurves for varying $v$ covers $X$. Applying what we showed above to this union implies $d = g$, a contradiction. 
\end{proof}

Recall that a multidegree $\un e$ is called effective, if it is non-negative on each irreducible component of $X$.

\begin{lma} \label{lma: semistable minus effective}
Let $X$ be a stable curve and $\ud$ a semistable multidegree of total degree $d \geq g$. Then there are effective multidegrees $\un e$  and $\un e'$ such that $\ud - \un e$ and $\ud - \un e'$ are semistable of total degree $g$ and $g-1$, respectively.
\end{lma}

\begin{proof}
We show the existence of $\un e$, that is, an effective multidegree such that $\ud - \un e$ is semistable of total degree $g$. Then $\md - \un e - v$ is semistable of degree $g-1$ by Lemma~\ref{lma: g-2 to g-1} for any vertex $v$ and thus we obtain also $\un e' \coloneqq \un e + v$ as claimed.
If $d = g$, there is nothing to show, so assume $d > g$ and let $L$ be a line bundle of multidegree $\md$. 

Since $\md$ is semistable, we have by definition that for all subcurves $Y \subset X$
\[
g(Y) - 1 +  (d - g + 1) \frac{2 g(Y) - 2 + |Y \cap Y^c|}{2g - 2} \leq \deg(L|_Y).
\]
By assumption we have $d - g + 1 > 1$ and since $X$ is stable, we have $0 < \frac{2 g(Y) - 2 + |Y \cap Y^c|}{2g - 2}$. In particular, we have for any proper subcurve $Y \subset X$
\[
g(Y) \leq \deg(L|_Y).
\]

Thus we can apply Lemma~\ref{lma:vertex to subtract from} and there is a vertex $v \in V(\GG_X)$ such that $g(Y) < \deg(L|_Y)$ whenever $X_v \subset Y$. Thus for every proper subcurve $Y$ of $X$ and smooth point $p \in X_v$ we have
\begin{equation} \label{eq:deg minus v}
    g(Y) \leq \deg\left(L|_Y(-p)\right).
\end{equation}

If $d = g + 1$, then $\md - v$ has total degree $g$ and Equation~\eqref{eq:deg minus v} gives that $\md - v$ is semistable. Otherwise, \eqref{eq:deg minus v} ensures that $\md - v$ again satisfies the assumption of Lemma~\ref{lma:vertex to subtract from}. We obtain $w$ such that $\md - v - w$ satisfies an inequality as in \eqref{eq:deg minus v}. Repeating this procedure $g - d$ times gives an effective multidegree $\un e \coloneqq v + w + \ldots$ such that $\md - \un e$ is semistable of total degree $g$, as claimed. 
\end{proof}

Since semistability is preserved in passing to residuals by Lemma~\ref{lma:semistable residual}, Lemma~\ref{lma: semistable minus effective} gives an analogous statement in degrees $d \leq g-2$. We will later use this version, and hence state it next.

\begin{lma} \label{lma: semistable plus effective}  Let $X$ be a stable curve and $\ud$ a semistable multidegree of total degree $d \leq g - 2$. Then there are effective multidegrees $\un e$  and $\un e'$ such that $\ud + \un e$ and $\ud + \un e'$ are semistable of total degree $g - 2$ and $g-1$, respectively. 
\end{lma}

The following example shows, that a claim analogous to Lemma~\ref{lma: semistable plus effective} does not hold for arbitrary total degrees. More precisely, it shows that for total degrees $d < d' < g-2$ and a semistable multidegree $\md$ of total degree $d$, there need not be an effective multidegree  $\un e$ of total degree $d' - d$ such that $\md + \un e$ is semistable.

\begin{ex} \label{ex:semistable plus effective}
Let $\GG_X$ be the graph with three vertices $v_1, v_2, v_3$ with two edges between $v_1$ and $v_2$ and two edges between $v_2$ and $v_3$. Let the weight of the vertices be $(2,1,2)$ and consider the multidegree $\md = (0, 3, 0)$. Then the total degree is $3 = g - 4$ and one checks that $\md$ is semistable. But $\md + v$ is not semistable for any vertex $v$ of $\GG_X$. Indeed, a semistable multidegree on $\GG_X$ of total degree $4$ needs to have value $1$ or $2$ on both $v_1$ and $v_3$.
\end{ex}

\subsection{Effective loci for semistable multidegrees} \label{subsec: general effective loci}

We are now ready to prove the main statement of this section. Recall that for a multidegree $\md$ we denote by $W_{\md}(X)$ the effective locus in $\Pic^{\md}(X)$, that is, the locus of line bundles $L$ with multidegree $\md$ and $h^0(X,L) \geq 1$. Its expected dimension is $d$, but it may very well be empty or of larger dimension.
Recall furthermore, that the multidegree $\md$ itself is called effective, if it is non-negative on each irreducible component.

\begin{thm} \label{thm:general effective loci}
Let $X$ be a stable curve and $\md$ a semistable multidegree of total degree $d \leq g - 2$. Then each irreducible component of $W_{\md}(X)$ has dimension at most $g - 2$.
If $\md$ is in addition effective, then the effective locus  $W_{\md}(X)$ contains an irreducible component of dimension $d$. 
\end{thm}

\begin{proof}
    Let $\md$ be a semistable multidegree of total degree $d \leq g - 2$. If $d < 0$, then the residual $\omega_X \otimes  L^{-1}$ has total degree greater than $2g - 2$ and is still semistable by Lemma~\ref{lma:semistable residual}. Hence $\omega_X \otimes L^{-1}$ is non-special by \cite[Theorem 2.3]{caporasosemistable} and thus $h^0(X,L) = 0$ by Lemma~\ref{lma:special residual}. So in this case $W_{\md}(X) = \emptyset$. Since a multidegree of negative total degree cannot be effective, this gives the claim.
    
    Assume now $d \geq 0$.
    By Lemma~\ref{lma: semistable plus effective}, there is a semistable multidegree $\md'$ of total degree $g - 2$ such that $\un e \coloneqq \md' - \md$ is effective. Fix a line bundle $[L_{\un e}] \in A_{\un e}(X)$ of multidegree $\un e$, where $A_{\un e}(X)$ denotes the image of the rational Abel map. Recall from Subsection~\ref{subsec:abel}, that this means $L_{\un e} = \cO_X\left(p_1 + \dots + p_{g - 2 - d}\right)$ for a collection of smooth points $p_i$, with $\un e_v$ of them contained in an irreducible component $X_v$ of $X$. Consider the isomorphism
    \[
    \varphi\colon \Pic^{\md}(X) \to \Pic^{\md'}(X), \; [L] \mapsto \left[L \otimes L_{\un e}\right].
    \]
    That is, $\varphi$ maps $L$ to $L\left(p_1 + \dots + p_{g - 2 - d}\right)$. In particular, $h^0(X, L) \leq h^0\left(X, L\left(p_1 + \dots + p_{g - 2 - d}\right)\right)$ and thus $\varphi(W_{\md}(X)) \subset W_{\md'}(X)$. Since $\md'$ is semistable of total degree $g-2$, $W_{\md'}(X)$ is either empty or of pure dimension $g - 2$  by Proposition~\ref{prop:dimension deg g - 2}. Hence the first claim follows. 
     
    For the second claim, we assume $\md$ is effective and thus we can consider the rational Abel map for $\md$. Its image $A_{\md}(X) \subset \Pic^{\md}(X)$ is irreducible and contained in an irreducible component $W$ of $W_{\md}(X)$.
    By Lemma~\ref{lma: semistable plus effective}, there is a semistable multidegree $\md''$, this time of total degree $g - 1$, such that $\un e' \coloneqq \md'' - \md$ is effective. Consider the map 
    \[
    \phi \colon W \times A_{\un e'}(X) \to V, \; ([L], [L_{\un e'}]) \mapsto [L \otimes L_{\un e'}].
    \]
    Here $V$ is an irreducible component of the effective locus $W_{\md''}(X)$ containing the image of $\phi$. Since $A_{\md}(X) \subset W$, we have $A_{\md''}(X) \subset V$ (note that $\md''$ is effective since $\md$ is). Thus $A_{\md''}(X) = V$ and $\dim(V) = g - 1$ by Proposition~\ref{prop:coelhoesteves}, since $\md''$ is semistable.
    In particular, the restriction of $\phi$ to $A_{\md}(X) \times A_{\un e'}(X)$ is surjective onto $V$.
    
    We claim that $\dim\left(A_{\un e'}(X)\right) = g - 1 - d$. Indeed, a general line bundle $[L''] \in A_{\md''}(X)$ satisfies $h^0(X, L'') = 1$ by Lemma~\ref{lma:coelhoesteves}. Since $h^0(X, L_{\un e'}) \leq h^0(X, L_{\un e'} \otimes L)$ for line bundles $[L_{\un e'}] \in A_{\un e'}(X)$ and $[L] \in A_{\md}(X)$, and the restriction of $\phi$ to $A_{\md}(X) \times A_{\un e'}(X)$ is surjective onto $A_{\md''}(X)$, it follows that $h^0(X, L_{\un e'}) = 1$ for a general line bundle $[L_{\un e'}] \in A_{\un e'}(X)$. Thus $\dim\left(A_{\un e'}(X)\right)$ equals the total degree of $\un e'$, $g - 1 - d$, by Lemma~\ref{lma:dimension Abel map}. 
    
    To finish the argument, choose $[L] \in W$ and $[L_{\un e'}] \in A_{\un e'}(X)$. The restrictions of $\phi$ to $\left\{[L]\right\} \times A_{\un e'}(X)$ and $W \times \left\{[L_{\un e'}]\right\}$ are injective. Denote by $W_{L}$ and $W_{L_{\un e'}}$ the respective images. By injectivity, we have $W_{L} \cap W_{L_{\un e'}} = \left \{[L \otimes L_{\un e'}] \right \}$. Combining these observations, we can use the usual estimate for the dimension of the intersection (see, for example, \cite[Proposition I.7.1]{Hartshorne}) to obtain \[0 = \dim\left(W_{L} \cap W_{L_{\un e'}}\right) \geq \dim(W) + \dim\left(A_{\un e'}\right) - \dim\left(A_{\md''}\right) = \dim(W) - d.\] Thus $\dim(W) \leq d$. On the other hand, $W$ is not empty, and hence $\dim(W) \geq d$ by Proposition~\ref{prop:determinantal}.
\end{proof}

\subsection{Further counterexamples} \label{subsec:further counterexamples}
We conclude with two examples, that exclude some possible strengthenings of Theorem~\ref{thm:general effective loci}. 

First, we give an example, where $\md$ is semistable, but the effective locus is irreducible and of dimension greater than the expected dimension. In particular, requiring that $\md$ is effective for the second claim in Theorem~\ref{thm:general effective loci} is necessary.

\begin{ex}\label{ex: no cpt of expected dim}
Let $\GG_X$ be the graph with two vertices $v_1$ and $v_2$, of respective weights $1$ and $5$, and three edges between them. Consider the multidegree $\md = (-1, 3)$, which one checks to be semistable (in fact, stable). Denote by $X_1$ and $X_2$ the irreducible components of $X$, corresponding to $v_1$ and $v_2$, respectively. The effective locus $W_{\md}(X)$ is then given as follows: any choice for $L|_{X_1}$, any choice of gluing data, and $L|_{X_2} = \mathcal O_{X_2}(X_1 \cap X_2)$. Thus $W_{\md}(X)$ is irreducible of dimension $3$, whereas  the total degree of $\md$ is $2$.
\end{ex}

Finally, we give an example where the multidegree $\md$ is semistable and effective, but $W_{\md}(X)$ has a component of dimension greater than the expected dimension. In particular, in the second claim of Theorem~\ref{thm:general effective loci} not all components need to be of expected dimension.

\begin{ex} \label{ex: g-3}
Let $\GG_X$ be the graph with two vertices $v_1$ and $v_2$, of respective weight $3$ and  $4$, and a single edge between them. Hence $g = 7$. Consider the multidegree $\md = (2, 2)$, which is effective and semistable (in fact, stable) of total degree $4 = g - 3$. Denote by $X_1$ and $X_2$ the irreducible components of $X$, corresponding to $v_1$ and $v_2$, respectively. Set $p = X_1 \cap X_2$. The effective locus $W_{\md}(X)$ has three irreducible components in this case: the first is the closure of $A_{\md}(X)$, which has dimension $4$, equal to the expected dimension. The second is $\Pic^2(X_1) \times \left \{\mathcal O_{X_2}(p + q)| q \in X_2 \right\}$, which again has dimension $4$. Finally, the third is $\left \{\mathcal O_{X_1}(p + q)| q \in X_1 \right\} \times \Pic^2(X_2)$ which has dimension $5 = g - 2$. 
\end{ex}

\providecommand{\bysame}{\leavevmode\hbox to3em{\hrulefill}\thinspace}
\providecommand{\MR}{\relax\ifhmode\unskip\space\fi MR }
\providecommand{\MRhref}[2]{%
  \href{http://www.ams.org/mathscinet-getitem?mr=#1}{#2}
}
\providecommand{\href}[2]{#2}

\end{document}